\documentclass{amsart}

\usepackage[all]{xy}
\usepackage{amsmath,amssymb, amsthm}
\usepackage{enumerate}
\usepackage{hyperref}
  \usepackage{graphicx}
  \usepackage{latexsym} 
  \usepackage{amsfonts} 
  \usepackage[2emode]{psfrag}
\usepackage{tikz, pgf}

\newtheorem{proposition}{Proposition}[section] 

\newtheorem{corollary}[proposition]{Corollary}
\newtheorem{theorem}[proposition]{Theorem}

\theoremstyle{definition}

\newtheorem{example}[proposition]{Example}
\newtheorem{examples}[proposition]{Examples}

\theoremstyle{remark}

\newtheorem{remark}[proposition]{Remark}

\newcommand{\thlabel}[1]{\label{th:#1}}
\newcommand{\thref}[1]{Theorem~\ref{th:#1}}
\newcommand{\selabel}[1]{\label{se:#1}}
\newcommand{\seref}[1]{Section~\ref{se:#1}}

\newcommand{\colabel}[1]{\label{co:#1}}
\newcommand{\coref}[1]{Corollary~\ref{co:#1}}

\newcommand{\exlabel}[1]{\label{ex:#1}}
\newcommand{\exref}[1]{Example~\ref{ex:#1}}

\newcommand{\eqlabel}[1]{\label{eq:#1}}
\newcommand{\equref}[1]{(\ref{eq:#1})}

\newcommand{\JVHom}{{\sf Hom}}
\newcommand{\fHom}{{\sf fHom}}

\newcommand{\JVEnd}{{\sf End}}
\newcommand{\JVExt}{{\sf Ext}}
\newcommand{\Fun}{{\sf Fun}}

\newcommand{\JVim}{{\sf Im}\,}

\newcommand{\can}{{\ul{\sf can}}}

\def\ot{\otimes}

\def\AA{{\mathbb A}}
\def\BB{{\mathbb B}}

\def\HH{{\mathbb H}}

\def\ZZ{{\mathbb Z}}

\newcommand{\Aa}{\mathcal{A}}

\newcommand{\Cc}{\mathcal{C}}
\newcommand{\Dd}{\mathcal{D}}

\newcommand{\Mm}{\mathcal{M}}

\newcommand{\Oo}{\mathcal{O}}

\newcommand{\Vv}{\mathcal{V}}
\newcommand{\Ww}{\mathcal{W}}

\newcommand{\Yy}{\mathcal{Y}}
\newcommand{\Zz}{\mathcal{Z}}

\def\text#1{{\rm {\rm #1}}}

\def\ol{\overline}
\def\ul{\underline}
\def\dul#1{\underline{\underline{#1}}}

\def\Set{\dul{\rm Set}}

\def\lim{{\rm lim\,}}

\def\vect{{\sf Vect}}

\def\Alg{{\sf Alg}}
\def\Coalg{{\sf Coalg}}

\def\Bialg{{\sf Bialg}}

\def\Hpfalg{{\sf Hpfalg}}

\def\M{{\sf M}}

\def\Fam{{\sf Fam}}
\def\Maf{{\sf Maf}}

\title{Hopf algebras---Variant notions and reconstruction
  theorems}
\author{Joost Vercruysse}
\address{D\'epartement de Math\'ematiques, Universit\'e Libre de Bruxelles, Boulevard du Triomphe, B-1050 Bruxelles, Belgium}
\email{jvercruy@ulb.ac.be}

\begin{document}

\begin{abstract}
Hopf algebras are closely related to monoidal categories. More precise, $k$-Hopf algebras can be characterized as those algebras whose category of finite dimensional representations is an autonomous monoidal category such that the forgetful functor to $k$-vectorspaces is a strict monoidal functor. This result is known as the Tannaka reconstruction theorem (for Hopf algebras). Because of the importance of both Hopf algebras 
in various fields, over the last last few decades, 
many generalizations have been defined. We will survey these different generalizations from the point of view of the Tannaka reconstruction theorem.
\end{abstract}

\maketitle

\section{Introduction}
Since the discovery of the notion of a Hopf algebra in the 1940s, they appeared as useful tools in relation with various fields of mathematics, such as number theory (the group ring of a formal group is a Hopf algebra), algebraic geometry (the algebra of regular functions on an algebraic group is a Hopf algebra, or more generally Hopf algebras are constructed from affine group schemes), Lie theory (the universal enveloping algebra of a Lie algebra is a Hopf algebra), Galois theory and separable field extensions (in relation with so-called Hopf-Galois theory), graded ring theory (one constructs a Hopf algebra from a graded ring and there is a strong relationship between graded modules and Hopf modules), locally compact group theory (quantum group theory), combinatorics (see e.g. \cite{AguMah:Coxeter} and the references therein), quantum mechanics and so on. 

It should be of no surprise that because of the large variety of applications of the theory of Hopf algebras, also a wide variety of mutations on the original definition occurred in the literature. Some of these variations are {\em quasi Hopf algebras} \cite{Drin:QH} that weaken certain (co)associativity constraints, {\em weak Hopf algebras} \cite{bohm:weakhopf} that weaken certain compatibility conditions on the (co)unit, {\em Hopf algebroids} \cite{BohmSzl:hgdax} that allow the transition to a non-commutative base, {\em multiplier Hopf algebras} \cite{VD:Mult} which are non-unital as algebra, {\em group Hopf co-algebras} \cite{Tur} that are a dual version of group graded of Hopf-algebras, and {\em Hopfish algebras} \cite{TWZ:Hopfish} that is a Morita invariant notion of Hopf algebras where structure maps are replaced by bimodules. 

Although the motivation to introduce these alternative notions was often quite diverse, it is surprisingly impressive how many features of the initial Hopf algebra theory can be transferred to each of the generalised versions. An explanation towards this striking fact might be offered by the use of monoidal categories. First of all, Hopf algebras, originally defined over a base field, can be defined in any braided monoidal category. Most of the theory can be quite easily lifted over to this setting, sometimes under additional assumptions such as the existence and preservation of certain (co)limits. Several, but not all, of the above mentioned generalizations can be understood as particular cases of Hopf algebras in a suitably chosen braided monoidal category. Maybe a more subtle approach, that we will use as a starting point in this survey, is not to treat Hopf algebras directly as objects with certain properties {\em in} a braided monoidal category, but to study them as monads {\em on} a monoidal category and characterize them by means of the Tannaka reconstruction (also known as Tannaka duality or Tannaka-Krein duality). 

The Tannaka reconstruction theorem originally stated that a compact
topological group is completely determined by its finite dimensional
representations. This result has been generalised to Hopf algebras and
large classes of quantum groups. Let $k$ be a field and $\vect_k^f$
the category of finite-dimensional $k$-vectorspaces. In this setting,
the Tannaka reconstruction theorem can be stated as follows. 
\index{Hopf algebra}
\index{Tannaka duality}
\begin{theorem}\thlabel{TannakaHopfalg}
There is a bijective correspondence between the following objects:
\begin{enumerate}[(i)]
\item $k$-Hopf algebras $H$
\item autonomous monoidal categories $\Vv$, together with a strict monoidal autonomous functor $U:\Vv\to \vect_k^f$.
\end{enumerate}
Under this correspondence, $\Vv\simeq \M_H^f$ (the category of finite-dimensional represenations of $H$) and $U$ is the usual forgetful functor.
\end{theorem}

After recalling some necessary notions from monoidal categories and basic Hopf algebra theory in Sections \ref{se:Prel} and \ref{se:Hopf}, it is our aim is to recall in \seref{Tannaka} the basic ideas of the Tannaka reconstruction theorem and some variations.  
We make a distinction between what we call ``simple reconstruction'', where the algebra (or dually coalgebra) object is already given by the onset, and the reconstruction concerns only the bialgebra or Hopf algebra structure in relation with the monoidal structure on the category of (co)representations (see \seref{simplemonad} for the case of monads and \seref{simplealgebra} for the algebra case); and what we could call ``the difficult part'' of the reconstruction. The latter concerns the ``real'' Tannaka reconstruction, and allows to reconstruct the algebra (or coalgebra) object itself from its category of (usually only finite-dimensional) (co)representations (see \seref{realTannaka}).

We will then study in \seref{Variations} how the variations on the notion of a Hopf algebra are related to variations of the Tannaka theorem, by changing the base category, the properties of the `forgetful' functor, or both, and show that the different generalizations of Hopf algebras can be re-obtained in a natural way. Of course, as this is a survey paper, the results that will be stated are not meant to be exclusively new. All (or most) theorems have appeared before, we will try to be as precise as possible with references, and we will refer for most proofs to these references as well.

There already exist several excellent surveys about Tannaka reconstruction, such as \cite{HewRos}, \cite{JoyStr:Tannaka}, \cite{Sch:Tannaka}. It makes no sense to repeat or copy this work here. We do not intend to explain very precisely how the reconstruction is obtained explicitly, and what the motivation originally was to do this (which is done perfectly in the references above), but we will use this construction to (hopefully) provide some insight in the `zoo' of Hopf-algebra-like structures.

Another survey that treats the different recent generalizations of Hopf algebras is \cite{Kar:survey}. The difference with the present paper is firstly the central role that we give to the Tannaka reconstruction theorem and secondly the fact that (for example) Multiplier Hopf algebras and group Hopf co-algebras are not considered in Karaali's paper. Nevertheless, we certainly can recommend this paper for a different point of view, especially for the treatment of quasi Hopf algebras, weak Hopf algebras and Hopf algebroids from the point of view of the dynamical quantum Yang Baxter equation, which is not considered here.

\section{Preliminaries}\selabel{Prel}

We suppose that the reader is familiar with the technicalities of
monoidal categories, as it is one of the main subjects of this
book. This section is only meant to fix the notation and
terminology. For more details, we refer to e.g.\ \cite{MacLane71}
Chapter VII and IX or to \cite{Chap11}. First of all, for an object $X$ in a category $\Cc$, we denote
the identity morphism on $X$ by $id_X$ or just by $X$. 
$$\xymatrix{X \ar[rr]^{id_X=X} && X}.$$

\subsubsection*{Monoidal categories}
\index{Category!monoidal}
Recall that a {\em monoidal category} $(\Cc,\ot,I,a,\ell,r)$ consists of a category $\Cc$, a functor $\ot:\Cc\times\Cc\to \Cc$, a monoidal unit object $I\in \Cc$, an associativity constraint (a natural isomorphism) $a_{X,Y,Z}:X\ot (Y\ot Z)\to (X\ot Y)\ot Z$ and unit constraints (natural isomorphisms) $\ell_X:I\ot X\to X$ and $r_X:X\ot I\to X$, satisfying suitable compatibility conditions. 
By Mac Lane's coherence Theorem, every monoidal category $(\Cc,\ot,I,a,l,r)$ is monoidally equivalent to a strict monoidal category $(\Cc',\ot',I',a',l',r')$, i.e. $a'$, $l'$ and $r'$ are identities, so there is no need to write them. As a consequence of this Theorem, we will omit also to write the data $a,l,r$ in the remaining (unless mentioned explicitly otherwise), a monoidal category will be shortly denoted by $(\Cc,\ot,I)$. We will make computations and definitions as if $\Cc$ was strict monoidal, however, by coherence, everything we do and prove remains valid in the non-strict setting (this will have important implications for quasi-Hopf algebras).

\subsubsection*{Braidings and symmetries}
\index{Category!braided monoidal} We call a monoidal category  $(\Cc,\ot,I,a,l,r)$
{\em braided} if
there exists a natural isomorphisms
$\gamma_{X,Y}:{X\ot Y}\to Y\ot X$, for all $X,Y\in \Cc$ satisfying
appropriate compatibility conditions with $a$, $l$ and $r$. 
If
$\gamma_{X,Y}^{-1}=\gamma_{Y,X}$ for all $X,Y\in\Cc$, then $\Cc$ is
said to be a {\em symmetric monoidal category}.
\index{Category!symmetric monoidal}

\subsubsection*{Rigidity}
\label{Rigid}
An object $X$ in a monoidal category is called {\em left rigid} if there exists an object $X^*$ together with morphisms $\eta:I\to X\ot X^*$ and $\epsilon:X^*\ot X\to I$ such that 
$$X\ot\epsilon\circ a^{-1}\circ \eta\ot X=X, \quad \epsilon\ot X^*\circ a\circ X^*\ot\eta= X^*$$
A {\em right rigid} object is defined symmetrically. A monoidal
category is said to be {\em left rigid} (resp. right rigid,
resp. rigid) if every object is left (resp. right, resp. left and
right) rigid. Another name for a rigid monoidal category is an {\em
  autonomous (monoidal) category}. If $\Cc$ is braided, then it is
right rigid if and only if it is left rigid.
\index{Category!autonomous}

\index{Category!closed monoidal}
A right {\em closed monoidal category} is a monoidal category $\Cc$ such that each endofunctor $X\ot -:\Cc\to \Cc$ associated to an object $X\in\Cc$ has a right adjoint $[X,-]$. A monoidal category is left closed if each endofunctor of the form $-\ot X$ has a right adjoint. Braided monoidal categories are left closed if and only if they are right closed, that is they are closed for short. If a category is right (resp. left) rigid, then it is right (resp. left) closed and $[X,-]\simeq X^*\ot -$.

\subsubsection*{Monoidal functors}
We warn the reader that our terminology of monoidal functors differs
slightly from the one used in \cite{Chap11} in this volume. What is called a monoidal functor in \cite{Chap11} is called a {\em strong} monoidal functor in this note. If one uses the terminology of \cite{Chap11}, then what is called a monoidal functor below, should be refered to as a {\em lax} monoidal functor.

\index{Functor!monoidal}
A functor $F:\Cc\to\Dd$ between the monoidal categories $(\Cc,\ot,I)$ and $(\Dd,\odot,J)$ is called a {\em monoidal functor} if there exists a $\Dd$-morphism $\phi_0:J\to F(I)$ and a natural transformation $\phi_{X,Y}:F(X)\odot F(Y)\to F(X\ot Y)$, $X,Y\in \Cc$, satisfying suitable compatibility conditions with relation to the associativity and unit constraints of $\Cc$ and $\Dd$. Furthermore, we make a difference between the notions of a {\em strong} monoidal functor $(F,\phi_0,\phi)$, where $\phi_0$ is an isomorphism and $\phi$ is a natural isomorphism and a {\em strict} monoidal functor $(F,\phi_0,\phi)$, where $\phi_0$ is the identity morphism and $\phi$ is the identity natural transformation. Dually, an {\em op-monodial functor} $F:\Cc\to \Dd$ is a functor for which there exists a morphism $\psi_0:F(I)\to J$ in $\Dd$   and morphisms $\psi_{X,Y}:F(X\ot Y)\to F(X)\odot F(Y)$ in $\Dd$, that are natural in $X, Y\in \Cc$, satisfying suitable compatibility conditions.
A strong monoidal functor $(F,\phi_0,\phi)$ is automatically op-monoidal. Indeed, one can take $\psi_0=\phi_0^{-1}$ and $\psi=\phi^{-1}$.
If $\Cc$ and $\Dd$ are braided monoidal categories, then a {\em braided monoidal functor} $F:\Cc\to \Dd$ is a monoidal functor such that $F\gamma_{X,Y}\circ \phi_{X,Y}=\phi_{Y,X}\circ \gamma_{FX,FY}:FX\odot FY\to F(Y\ot X)$.
Of course, one has a canonical definition of a {\em monoidal natural transformation} between monoidal functors.

Monoidal functors behave nicely with respect to adjuctions, as can be seen from the following classical theorem.

\begin{theorem}
Let $F:\Cc\to\Dd$ be a left adjoint to $G:\Dd\to \Cc$, where $(\Cc,\ot,I)$ and $(\Dd,\odot,J)$ are monoidal categories. Then $F$ is op-monoidal if and only if $G$ is a monoidal functor. 
\end{theorem}

If $(F,G)$ is an adjoint pair of monoidal functors such that the unit and counit of this adjunction are monoidal natural transformations then we call $(F,G)$ a {\em monoidal adjunction}. By the above lemma, the left adjoint of a monoidal adjucntion is also op-monodial, in fact it is even strong monoidal.

In a natural way, one defines a monoidal functor between rigid monoidal categories to be rigid (or autonomous), if it preserves dual objects. A monoidal functor between (right, left) closed monoidal categories is said to be (right, left) closed if it commutes with the adjoints of the endofunctors of type $X\ot -$ .

\section{Bialgebras and Hopf algebras in monoidal categories}\selabel{Hopf}

\subsection{Algebras and coalgebras with their representations}

\index{Algebra}
\index{Monoid}
\index{Coalgebra}
\index{Comonoid}
Let $\Cc=(\Cc,\ot,I)$ be a monoidal category. An {\em algebra} or {\em monoid}
in $\Cc$ is a triple
$A=(A,m,u)$, where $A\in \Aa$ and $m: A\ot A\to A$ and $u: I\to A$ are
morphisms in $\Cc$ such that the following diagrams commute:
$$
\xymatrix{
A\ot A\ot A\ar[rr]^{m\ot A}\ar[d]_{A\ot m}&& A\ot A\ar[d]^{m}\\
A\ot A\ar[rr]^{m}&&A}~~~~~~
\xymatrix{I\ot A\ar[r]^{u\ot A}\ar[rd]_{\cong}&
A\ot A\ar[d]^{m}&A\ot I\ar[l]_{A\ot u}\ar[dl]^{\cong}\\
&A&}
$$
A {\em coalgebra} or {\em comonoid} in a monoidal category $\Aa$ is an algebra
in the opposite category $\Cc^{\rm op}$. 

\index{Module}
\index{Representation}
A right {\em module} or {\em representation} over an algebra $A$ in $\Cc$ is a pair $(M, \rho_M)$ where $M$ is an object of $\Cc$ and $\rho_M:M\ot A\to M$ is a morphism in $\Cc$ satisfying the following associativity and unitary condition
\[
\xymatrix{
M\ot A\ot A\ar[rr]^{\rho_M\ot A}\ar[d]_{M\ot m}&& M\ot A\ar[d]^{\rho_M}\\
M\ot A\ar[rr]^{\rho_M}&&M} \quad
\xymatrix{
M \ot A \ar[rr]^{\rho_M} && M \\
M\ot I \ar[u]^-{M\ot u} \ar[urr]_-{\cong}
}
\]
A left module $(M,\lambda_M)$ is defined in a symmetric way by means of a left action $\lambda:A\ot M\to M$.

\index{Comodule}
A right {\em comodule} $(M,\rho^M)$ over a coalgebra $C$ in $\Cc$ is a right module over $C$, considered as an algebra in $\Cc^{\rm op}$. In particular, the morphism $\rho^M:M\to M\ot C$ in $\Cc$ is called the coaction.

Morphisms of algebras, coalgebras, modules and comodules are defined in an obvious way as structure-preserving morphisms from $\Cc$. This leads to the introduction of the categories $\Alg(\Cc)$, $\Coalg(\Cc)$, $\Cc_A$ and $\Cc^C$ of respectively algebras in $\Cc$, coalgebras in $\Cc$, right modules over a fixed algebra $A$ in $\Cc$ and right comodules over a fixed coalgebra in $\Cc$.

\begin{example}\exlabel{3.1.1}
Let $k$ be a commutative ring. An algebra in $\Mm_k$ is a $k$-algebra, and
a coalgebra in $\Mm_k$ is a $k$-coalgebra. Modules and comodules are the classical ones.
\end{example}

\begin{example}\exlabel{3.1.2}
An algebra in $\Set$ is a monoid $(G,m,u)$. A module over $G$ is a $G$-set. 
Every set $X$ has a unique structure of a coalgebra in $\Set$. The comultiplication $\Delta: X\to X\times X$ is
the diagonal map $\Delta(x)=(x,x)$, and the counit is the unique map $\varepsilon:X\to \{\ast\}$. A comodule over $X$ is then a set $Y$ together with a (any) map $f:Y\to X$. The comultiplication $\rho_f:Y\to Y\times X$ is given by $\rho_f(y)=(y,f(y))$.
\end{example}

\begin{example}\exlabel{Ccategory}
Consider the category of categories (where of course some care has to be taken in to account with respect to the kind of ``largeness'' of the categories one is considering, to overcome set-theoretic problems, but we omit this discussion here), with Cartesian product of categories as tensor product. An algebra in this category is a (strict) monoidal category. A module over a monoidal category $\Cc$, is a category $\Mm$ together with a bifunctor
$\ot_\Mm:\Mm\times \Cc\to \Mm$, such that we have natural transformations
$M\ot_\Mm (X\ot Y) \cong (M\ot_\Mm X)\ot_\Mm Y$ and $M\ot_\Mm I\cong M$ that satisfy a suitable collection of compatibility conditions. We call such a category a right $\Cc$-category. 
\end{example}

\exref{Ccategory} provides us with a tool to consider a wider range of modules and comodules. Let $(A,m,u)$ be an algebra in the monoidal category $\Cc$ and let $(\Mm,\ot_\Mm)$ be a right $\Cc$-category. Then a right $A$-module in $\Mm$ is an object $M\in\Mm$ endowed with a morphism $\rho_M:M\ot_\Mm A\to M$ in $\Mm$ satisfying the following associativity and unitary contraints
\[
\xymatrix{
(M\ot_\Mm A)\ot_\Mm A\ar[rr]^{\rho_M\ot_\Mm A}\ar[d]_\cong&& M\ot_\Mm A\ar[dd]^{\rho_M}\\
M\ot_\Mm (A\ot A) \ar[d]_{M\ot_\Mm m} \\
M\ot_\Mm A\ar[rr]^{\rho_M}&&M} \quad
\xymatrix{
M \ot_\Mm A \ar[rr]^{\rho_M} && M \\
M\ot_\Mm I \ar[u]^-{M\ot_\Mm u} \ar[urr]_-{\cong}
}
\]

\begin{example}\exlabel{3.1.3}
Let $A$ be an associative ring with unit. The category of (unital) $A$-bimodules ${}_A\Mm_A$ is a monoidal
category; the monoidal product is given by the tensor product $\ot_A$ over $A$, and the unit object by $A$.
An algebra in ${}_A\Mm_A$ is just a ring morphism $\iota:A\to R$. A coalgebra in 
${}_A\Mm_A$ is called an $A$-coring (see e.g.\ \cite{BrzWis:book}). One can easily check that $\Mm_A$ is a right ${_A\Mm_A}$-category, where we use again the tensor product over $A$. So we can compute (co-)modules over an $A$-(co)ring in $\Mm_A$.
\end{example}

\begin{example}
Monads and comonads are algebras and coalgebras in a monoidal category of endofunctors ${\sf EndoFun}(\Cc)$ on a category $\Cc$. I.e.\ a monad is an endofunctor $\AA:\Cc\to\Cc$, with natural transformations $m:\AA\AA\to \AA$ and $u:id_\Cc\to \AA$ satisfying suitable unitary and associativity conditions.
Obviously, $\Cc$ is a left ${\sf EndoFun}(\Cc)$-category. A module over a monad $(\AA,m,u)$ in $\Cc$ called an Eilenberg-Moore object. Explicitly, this is a pair $(X,\rho_X)$, where $X\in \Cc$ and $\rho_X:\AA(X)\to X$ is a morphism in $\Cc$, satisfying the obvious associativity and unitary conditions. We denote the category of Eilenberg-Moore objects (EM-category for short) by $\Cc_\AA$. If $\BB$ is a comonad, then we denote the category of comodules over $\BB$ in $\Cc$ (also called the category of Eilenberg-Moore objects) by $\Cc^\BB$.

If $(A,m,u)$ is an algebra {\em in} a monoidal category $\Cc$, then this induces canonically a monad $(\AA,M,U)$ {\em on} $\Cc$ by defining $\AA=-\ot A$, $M=-\ot m:\AA\AA=-\ot A\ot A\to \AA=-\ot A$ and $U\simeq u$ by putting for any $X\in \Cc$, $U_X=\xymatrix{X\cong X\ot I\ar[r]^-{X\ot u} & X\ot A = \AA X}$. The EM-category of $\AA$ is exactly the category of right $A$-modules.
\end{example}

\subsection{Bialgebras and Hopfalgebras with their representations}

Let $(\Cc,\ot,I,\gamma)$ be a braided monoidal category.
Then $\Alg(\Cc)$ and $\Coalg(\Cc)$ are monoidal categories with a strict monoidal forgetful functor to $\Cc$. Indeed, for any two algebras $(A,m_A,u_A)$ and $(B,m_B,u_B)$, we can construct a new algebra
$$(A\ot B,m=(m_A\ot m_B)\circ (A\ot \gamma_{B,A}\ot B), u_A\ot u_B).$$

A {\em bialgebra} in $\Cc$ is a coalgebra in the monoidal category $\Alg(\Cc)$ of algebras or equivalently, an algebra in the monoidal category of coalgebras $\Coalg(\Cc)$. Explicitly, a bialgebra in $\Cc$ is an object $B$ enriched with an algebra structure $(B,m,u)$, and a coalgebra structure $(B,\Delta,\epsilon)$ such that any of the following sets of equivalent conditions hold
\begin{itemize}
\item $\Delta$ and $\epsilon$ are algebra maps;
\item $m$ and $u$ are coalgebra maps;
\item the following diagrams are commutative
\begin{eqnarray}\eqlabel{bialgebra}
&\xymatrix{
B \ot B \ar[rrr]^-{(B\ot \gamma_{B,B}\ot B)\circ \Delta\ot \Delta} \ar[d]_-m &&& B\ot B\ot B\ot B \ar[d]^-{m\ot m}\\
B \ar[rrr]_-\Delta &&& B\ot B
} \qquad
\xymatrix{
I\ar[r]^-u & B\ar@{=}[d]\\
& B \ar[ul]^-{\epsilon}
}\\
&\xymatrix{
B\ar[rr]^-{\Delta} && B\ot B\\
I \ar[u]^u \ar[rr]^\cong && I\ot I \ar[u]_-{u\ot u}
} \qquad
\xymatrix{
B\ot B\ar[d]_{\epsilon\ot\epsilon} \ar[rr]^-{m} && B \ar[d]^-{\epsilon}\\
I\ot I  \ar[rr]^\cong && I 
}\nonumber
\end{eqnarray}
\end{itemize}
\index{Bialgebra}

An {\em antipode} for a bialgebra $H$ in $\Cc$ is a $\Cc$-morphism $S:H\to H$ such that 
$$m\circ (H\ot S)\circ \Delta = u\circ \epsilon =m\circ (S\ot H)\circ \Delta.$$
In fact, this means that $S$ is an inverse for the identity map on $H$ in the monoid $\JVEnd_\Cc(H)$ for the convolution product $\ast$ that is defined by
$$f\ast g=m\circ (f\ot g)\circ \Delta\in \JVEnd_\Cc(H),$$
for any $f,g\in\JVEnd_\Cc(H)$.
A {\em Hopf algebra} in $\Cc$ is a bialgebra that possesses an antipode. 
An immediate property of the antipode tells us that it is an anti-algebra morphism and an anti-coalgebra morphism.
Some authors only consider Hopf algebras with invertible antipodes, here we will mention explicitly when the antipode is invertible (with respect to usual composition).
\index{Hopf algebra}

Morphisms of bialgebras are $\Cc$-maps that preserve both the algebra and coalgebra structure. Morphisms of Hopf algebras are morphisms of the underlying bialgebras. It can be proven that Hopf algebra morphisms preserve also the antipode. Hence we constructed the category $\Bialg(\Cc)$ of bialgebras in $\Cc$ and its full subcategory $\Hpfalg(\Cc)$ of Hopf algebras in $\Cc$.

For a bialgebra (hence also for a Hopf algebra), we can consider its category of right (resp. left) modules and right (resp. left) comodules. Another interesting category is the category of Hopf modules. A (right,right) {\em Hopf module} over a bialgebra $H$ is a triple $(M,\rho_M,\rho^M)$, where $(M,\rho_M)$ is a right $H$-module, $(M,\rho^M)$ is a right $H$-comodule, and the following compatibility condition holds
$$(\rho_M\ot m)\circ (M\ot \gamma_{H,H}\ot H)\circ (\rho^M\ot \Delta)=\rho^M\circ \rho_M:M\ot H\to M\ot H.$$
Morphisms of $H$-Hopf modules are $\Cc$-morphisms that are at the same time $H$-module morphisms and $H$-comodule morphisms. The category of (right,right) $H$-Hopf modules is denoted by $\Cc^H_H$. For any bialgebra, we can consider the functor
\begin{equation}
\eqlabel{comparisionHopf}
-\ot H:\Cc\to \Cc^H_H
\end{equation}
If $\Cc$ admits equalizers, then this functor has a right adjoint that we denote by $(-)^{co H}$, the functor that takes $H$-coinvariants and that is computed by the following equalizer at any $M\in \Cc^H_H$
$$\xymatrix{ M^{co H}\ar[rr] && M \ar@<.5ex>[rr]^-{\rho^M} \ar@<-.5ex>[rr]_-{M\ot u} &&M\ot H}.$$ If $H$ is a Hopf algebra, the fundamental theorem for Hopf modules says that this functor is an equivalence of categories \cite{Tak:braidedhopf}. The fact that this functor is an equivalence of categories is even equivalent with the bialgebra $H$ being a Hopf algebra. 
\begin{theorem}[Fundamental theorem for Hopf modules]\thlabel{fundamental}
Let $H$ be a bialgebra in a braided monoidal category $\Cc$ with equalizers, then the following statements are equivalent.
\begin{enumerate}[(i)]
\item the functor $(-)^{co H}:\Cc^H_H\to \Cc$ is fully faithful;
\item the pair $(-\ot H,(-)^{co H})$ is an equivalence of categories between $\Cc$ and $\Cc_H^H$;
\item $\can=(m\ot H)\circ (H\ot \Delta):H\ot H\to H\ot H$ is a $\Cc$-isomorphism;
\item $H$ admits an antipode, i.e.\ $H$ is a Hopf algebra in $\Cc$.
\end{enumerate}
\end{theorem}

It is generally known and easily verified that a monoidal functor $F:\Cc\to \Dd$ sends an algebra $A$ in $\Cc$ to an algebra $F(A)$ in $\Dd$. Similarly, an op-monoidal functor sends a coalgebra $C$ in $\Cc$ to a coalgebra $F(C)$ in $\Dd$. Finally, a strong monoidal functor sends a bialgebra $B$ (respectively a Hopf algebra $H$) to a bialgebra $F(B)$ (respectively a Hopf algebra $F(H)$) in $\Dd$.

\begin{examples}
\begin{enumerate}[(1)]
\item As in $\Set$ all objects are in a unique way coalgebras, $\Coalg(\Set)=\Set$. Consequently every monoid, i.e. every algebra in $\Set$, also an algebra in $\Coalg(\Set)$, hence a bialgebra in $\Set$. 
A Hopf algebra in $\Set$ is nothing else than a group.
\item
A bialgebra, resp. a Hopf algebra, in $\Mm_k$ is a classical $k$-bialgebra, resp. Hopf $k$-algebra.
As the linearisation functor $k-:\Set\to \Mm_k$ is a strong monoidal functor, every group algebra $kG$ is in a canonical way a (cocommutative) Hopf algebra. On the other side, the contravariant functor $\Fun(-,k):\Set\to \Mm_k$, which sends a set $X$ to the vectorspace of $k$-valued functions on $X$, is only a monoidal functor (if we consider it as a covariant functor from $\Set^{op}\to \Mm_k$). Hence $\Fun(X,k)$ has the structure of a (commutative) algebra for any set $X$. 
If we consider only finite sets, then $\Fun(-,k):\Set^f\to\Mm_k$ becomes again strong monoidal and hence $\Fun(G,k)$ Hopf algebra is a (commutive) Hopf algebra if $G$ is a finite group,
but in general $\Fun(G,k)$ is no longer a Hopf algebra when $G$ is an infinite group. However, if we consider the more restrictive functor of regular functions on affine sets (spaces) $\Oo:{\sf Aff}\to \Mm_k$, then we obtain again a strong monoidal functor. As the algebra structure is obtained from the unique coalgebra structure on the affine set given by the diagonal map, this ``explains'' why regular functions on an affine space gives rise to a commutative algebra. Moreover if $G$ is an affine group, then $\Oo(G)$ is a Hopf algebra that is commutative as algebra. By deforming the multiplicative structure of these algebras, one can construct non-commutative non-cocommutative Hopf algebras that are called quantum groups, see e.g. \cite{kassel}. (Often, quantum groups are rather considered from the point of view of differential geometry than from algebraic geometry, considering Lie groups rather than affine groups.)
\item
Let $\Mm_k^{\ZZ_2}$ be the category of $\ZZ_2$-graded $k$-modules. This is a symmetric monoidal category with a symmetry given by 
$$\sigma_{X,Y}:X\ot Y\to Y\ot X,\ \sigma_{X,Y}(x\ot y)=(-1)^{|x||y|}y\ot x,$$
where $|\cdot |$ denotes the degree of an element. A Hopf algebra in this category is a Hopf superalgebra.
\end{enumerate}
\end{examples}

\subsection{Bimonads and Hopf monads}

The word  bimonad (and Hopf monad) is (are) in use for different notions that are all meant to generalize bialgebras (or Hopf algebras) in the setting of monads.
The idea of a bimonad in the sense of \cite{BLV} and \cite{Moe} (in the last cited paper the author  
uses, somewhat misleading, the term `Hopf monad' for what is called a bimonad 
in the first cited paper and also below) is to transfer the notion of a bialgebra from the setting of braided monoidal categories to usual monoidal categories, similar to the case of usual monads, that made it possible to transfer the notion of an algebra from the setting of monoidal categories to arbitrary categories. The notion of a bimonad in the sense of Mesablishvili and Wisbauer \cite{MesWis} has the ambition to go even a step further and wants to free the notion of a bialgebra even of the monoidal structure. The price to pay is that reconstruction theorems are not as easily obtained in this last setting, but some other Hopf-algebraic features, such as the fundamental theorem arise more naturally. So both approaches have there advantages and disadvantages. Because this paper is greatly motivated by the reconstruction theorems, we will advocate the approach of Brugui\`eres-Lack-Virelizier and Moerdijk.

A {\em bimonad} $\BB=(\BB,M,U)$ on a monoidal category $\Cc$ is an op-monoidal monad, that is $\BB$ is a monad on $\Cc$ such that the underlying functor $\BB$ is an op-monoidal functor and the natural transformations $M:\BB\BB\to\BB$ and $U:id_\Cc\to \BB$ are opmonoidal natural transformations. In particular, $\BB$ is endowed with a natural transformation $\phi^\BB_{X,Y}:\BB(X\ot Y)\to \BB X\ot \BB Y$ and a morphism $\phi:\BB I\to I$ satisfying suitable compatibility conditions (see e.g. \cite[Definition 2.4]{BLV}, we will omit the technicalities in this review). 
A {\em bicomonad} is defined as a monoidal comonad.

The {\em right fusion operator} of a bimonad $\BB$ is defined as the natural transformation
$$\can_{X,Y}^r=(M_X\ot \BB Y)\circ \phi^\BB_{\BB X,Y}  : \BB(\BB X\ot Y)\to \BB X\ot \BB Y $$
Similarly one defines a left fusion operator $\can^\ell:\BB(1_\Cc\ot \BB)\to \BB\ot \BB$.

\index{Hopf monad}
A right (resp. left) {\em Hopf monad} is a bimonad such that its right (resp. left) fusion operator is a natural isomorphism. A Hopf monad is at the same time a left and right Hopf monad.

\begin{example}\exlabel{bialgebrabimonad}
Any bialgebra $(B,m,u,\Delta,\epsilon)$ in a braided monoidal category $\Cc$ induces bimonad $\BB=-\ot B$ on $\Cc$. Let us just give the formula for the structural natural transformation $\phi^\BB$ and the morphism $\phi$,
\begin{eqnarray*}
\phi^\BB_{X,Y}&=&(X\ot \gamma_{Y,B}\ot B)\circ X\ot Y\ot \Delta:(X\ot Y)\ot B\to (X\ot B)\ot (Y\ot B);\\
\phi&=&\epsilon:I\ot B\cong B\to I.
\end{eqnarray*}
Moreover, any Hopf algebra $H$ induces right Hopf monad $\HH=-\ot H$ on $\Cc$. In this situation, the right fusion operator is given explicitly by 
\begin{eqnarray*}
&&\can^{\HH,r}_{X,Y}=(X\ot \gamma_{Y,H}\ot H)\circ (X\ot Y\ot \can)\circ (X\ot \gamma_{Y,H}^{-1}\ot H):\\
&&\hspace{6cm}X\ot H \ot Y\ot H\to X\ot H\ot Y\ot H
\end{eqnarray*}
where $\can$ is the canonical map that appears in \thref{fundamental}. In a similar way, the left fusion operator is in correspondence with the morphism $\can\circ\gamma_{H,H}$, where $\gamma_{H,H}$ denotes the braiding on the Hopf algebra $H$. If the antipode is invertible, then $\HH$ is also a left Hopf monad, hence a Hopf monad.
\end{example}

Recall that an object $G\in \Cc$ is called a generator if and only if the functor $\JVHom_\Cc(G,-):\Cc\to \Set$ is fully faithful. If the category $\Cc$ has coproducts, this is furthermore equivalent with the fact that for any object $X\in \Cc$ there is a canonical epimorphism $f_X:H=\coprod_{f:G\to X} G \to X$, where the coproduct takes over a number of copies of $G$. Therefore, we find a fork 
\begin{equation}\eqlabel{coequalizergenerator}\xymatrix{\coprod_{^{(g,h):G\to H,\ st}_{f_X\circ g=f_X\circ h}} G \ar@<.5ex>[rr]^-{g_X} \ar@<-.5ex>[rr]_-{h_X} &&\coprod_{f:G\to X} G \ar[rr]^{f_X} &&  X}\end{equation}
In general this diagram is not a coequalizer, but $G$ is called a {\em regular} generator if \equref{coequalizergenerator} is a coequalizer for every $X\in \Cc$, see e.g. \cite[page 81]{Kelly}. If we work moreover in a category where the endofunctors $-\ot X$ and $X\ot -$ preserve colimits, it is not hard to proof that for any two objects $X$ and $Y$, one obtains a canonical epimorpfism of the form $f_X\ot f_Y:\coprod G\ot G\to X\ot Y$, hence two morphisms $h,h':X\ot Y\to Z$ are identical in $\Cc$ if and only if $h\circ (f\ot g)=h'\circ (f\ot g)$ for all $f:G\to X$ and $g:G\to Y$.

We now state a first simple `reconstruction-type' theorem.

\begin{theorem}\thlabel{pretannakabialgebra}
Let $\Cc$ be a cocomplete braided monoidal category such that $I$ is a regular generator of $\Cc$ and that $-\ot X$ and $X\ot -$ preserve colimits in $\Cc$, for all $X\in \Aa$. Let $(B,m,u)$ be an algebra in $\Cc$ and $(\BB,M,U)$ the associated monad on $\Cc$. Then there is a bijective correspondence between
\begin{enumerate}[(i)]
\item bialgebra structures on $B$ (in $\Cc$) with underlying algebra $(B,m,u)$;
\item bimonad structures on $\BB$ (on $\Cc$) with underlying monad $(\BB,M,U)$.
\end{enumerate}
Moreover, under the above correspondence, there is furthermore a bijective correspondence between antipodes turning $B$ into a Hopf algebra and inverses of the right fusion operator turning $\BB$ into a right Hopf monad. Moreover, there is a bijective correspondence between invertible antipodes on $B$ and inverses for both the left and right fusion operators turing $\BB$ into a Hopf monad.
\end{theorem}

\begin{proof}
We already know from \exref{bialgebrabimonad} that a bialgebra $B$ induces a bimonad $\BB=-\ot B$. Conversely, if the monad $\BB=-\ot B$ is a bimonad, then one can easily verify that the morphisms $\Delta=\phi^\BB_{I,I}:B\cong (I\ot I)\ot B\to (I\ot B)\ot (I\ot B)\cong B\ot B$ and $\epsilon=\phi: B\cong I\ot B\to I$ define a comultiplication and counit on $B$ that turn $B$ into a bialgebra. 

Let us prove that these constructions give a bijective correspondence. For any two objects $X, Y\in \Cc$, and any two morphisms $f:I\to X$ and $g:I\to Y$, we find by naturality of $\phi^\BB$
\begin{eqnarray*}
\phi^\BB_{X,Y}\circ (f\ot g\ot B)&=&(f\ot B\ot g\ot B)\circ \phi^\BB_{I,I}
\end{eqnarray*}
Moreover, by the properties of $I$ as monoidal unit in $\Cc$, we can consider $\phi^\BB_{I,I}:B\to B\ot B$ and it follows furthermore that 
$$(f\ot B\ot g\ot B)\circ \phi^\BB_{I,I}=(X\ot \gamma_{Y,B}\ot B)\circ (X\ot Y\ot \phi^\BB_{I,I})\circ (f\ot g\ot B)$$
As this holds for all morphisms $f,g$ and $I$ is a regular generator, we obtain that (see remark above)
$$\phi^{\BB}_{X,Y}=(X\ot \gamma_{Y,B}\ot B)\circ (X\ot Y\ot \phi^\BB_{I,I}).$$
In the same way, one proofs that $\epsilon=\phi$.

For the last statement, we know from \thref{fundamental} that antipodes on the bialgebra $B$ are in correspondence with the bijectivity of the canonical map $\can:B\ot B\to B\ot B$. Furthermore, as mentioned in \exref{bialgebrabimonad}, the fusionoperator is up to natural isomorphism completely determined by the canonical map. Hence the fusion operator is invertible if and only if the canonical map is invertible if and only if $B$ has an antipode. Finally, it known that both $\can$ and $\can\circ\gamma_{H,H}$ are bijective if and only if the antipode of a Hopf algebra is invertible.
\end{proof}

\begin{remark}
In Hopf algebra theory (over a base field), one considers sometimes Hopf algebras with a one-sided antipode. That is, a {\em right} antipode for a bialgebra $H$, is a morphism $S:H\to H$ that is a {\em right} inverse of the identity with relation to the convolution product. A right Hopf algebra is then a bialgebra that possesses a right antipode. Left antipodes and left Hopf algebras are defined in the same way. Let us remark that this terminology does not fully correspond to the case of Hopf monads. As we have seen, a {\em right} Hopf monad is a bimonad, such that the {\em right} fusion operator is a natural isomorphism. By the above theorem, if the monad $\HH$ is of the form $ -\ot H$ for an algebra $H$, then right Hopf monad structures on $\HH$ correspond exactly to bialgebra structures on $H$ with a usual antipode, that is an ordinary Hopf algebra structures on $H$ (not just right). Similarly, a left Hopf monad of the form $H\ot -$ corresponds as well to a Hopf algebra structure on $H$. So $H\ot -$ is a left Hopf monad if and only if $-\ot H$ is a right Hopf monad if and only if $H$ is a Hopf algebra.
In a similar way as \thref{pretannakabialgebra}, one proves that a (left and right) Hopf monad of the form $-\ot H$ corresponds exactly to a Hopf algebra $H$ with invertible antipode.
\end{remark}

\begin{remark}
As mentioned above, Mesablishvili and Wisbauer \cite{MesWis} introduced an alternative notion of bimonad and Hopf-monad. In their work a bimonad on a (not necessarily monoidal) category $\Cc$ consists of a monad $(\BB,m,u)$ on $\Cc$ such that the underlying functor is endowed at the same time with the structure of a comonad $(\BB,\Delta,\epsilon)$ on $\Cc$, and there is a mixed distributive law $\lambda:\BB\BB\to \BB\BB$ relating the monad and comonad structure. This setting allows to consider a suitable kind of Hopf modules $\Cc^\BB_\BB$ along with a comparison functor $\Cc\to \Cc^\BB_\BB$, similar to the functor \equref{comparisionHopf}. An antipode can now be introduced similarly as in the case of usual Hopf algebras as a natural transformation $S:\BB\to \BB$ satisfying $m*(SH)*\Delta=u*\epsilon=m*(HS)*\Delta$ (where $*$ is the Godement product). \thref{fundamental} has a very natural generalization in this setting. (A version of the fundamental theorem also exists in the setting of Hopf monads by Brugui\`eres, Lack and Virelizier.)
\end{remark}

\section{Reconstruction theorems}\selabel{Tannaka}

\subsection{Monoidal structure on the representation categories of bimonads and Hopf monads}
\selabel{simplemonad}

As mentioned earlier, it is classically known that bialgebras can be characterised as those algebras whose category of modules is a monoidal category with a strict monoidal forgetful functor. The monadic version of this theorem is of folkloristic knowledge. For a sligthly different formulation of the following theorem, see \cite[Theorem 7.1]{Moe} or \cite[section 2.4]{Szl:monoidalbialgebroid}.

\begin{theorem}
\thlabel{Tannakabimonad}
Let $(\BB,M,U)$ be a monad on a monoidal category $\Cc$. Then there is a bijective correspondence between 
\begin{enumerate}[(i)]
\item bimonad structures on the functor $\BB$ with underlying monad $(\BB,M,U)$; and 
\item monoidal structures on the Eilenberg-Moore category of the monad $(\BB,M,U)$ such that the forgetful functor to $\Cc$ 
is a strict monoidal functor.
\end{enumerate}
\end{theorem}

The following theorem follows now by a duality argument.

\begin{theorem}
Let $(\BB,M,U)$ be a monad on a monoidal category $\Cc$. Then there is a bijective correspondence between 
\begin{enumerate}[(i)]
\item stuctures of a monoidal monad on the functor $\BB$ with underlying monad $(\BB,M,U)$; and
\item monoidal structures on the Kleisli category of the monad $(\BB,M,U)$ such that the forgetful functor to $\Cc$ is a strict monoidal functor. 
\end{enumerate}
\end{theorem}

As will become clear when we proceed, to characterize Hopf algebras, we need rigidity conditions on our underlying category. A first theorem that already fully characterises Hopf monads in a very beautiful and general way is the following theorem, that is a reformulation of {\cite[Theorem 3.6]{BLV}}, where we only ask for a closed monoidal category (such as a module category). Although it is not as strong as the original Tannaka theorem from the reconstruction point of view, it explains very nicely what the internal meaning is of Hopf monads, and therefore Hopf algebras, in terms of monoidal categories.

\begin{theorem}[{\cite[Theorem 3.6]{BLV}}]
\thlabel{TannakaHopfmonad}
Let $(\BB,M,U)$ be a monad on a right closed monoidal category $\Cc$. Then there is a bijective correspondence between 
\begin{enumerate}[(i)]
\item right Hopf monad structures on $\BB$ with underlying monad $(\BB,M,U)$; and
\item right closed monoidal structures on the Eilenberg-Moore category of the monad $(\BB,M,U)$ such that the forgetful functor to $\Cc$ is right closed and strict monoidal.
\end{enumerate}
\end{theorem}

In the classical theory of Hopf algebras, rather than considering a closed structure on the category of (right) $H$-modules over a Hopf algebra $H$, one often uses the antipode to put a right $H$-module structure on the dual space $M^*$ of any right $H$-module $M$ as follows
$$f\cdot h=f(e_i\cdot S(h))f_i, $$
where $f\in M^*$, $h\in H$ and $\{e_i,f_i\}\in M\times M^*$ is a finite dual basis. This property follows now directly from \thref{TannakaHopfmonad}.

\begin{corollary}
\colabel{TannakaHopfmonadif}
Let $\Cc^f$ be the right rigid full subcategory of the closed monoidal category $\Cc$. If $\BB$ is a Hopf monad on $\Cc$, then there is a right rigid full subcategory $\Cc^f_\BB$ of the EM-category of $\BB$ such that the forgetful functor restricts and corestricts to a functor $U^f:\Cc^f_\BB\to \Cc^f$ that is strict monoidal (and hence right rigid).
\end{corollary}

\begin{proof}
One can construct the category $\Cc^f_\BB$ as the pullback of the functor $U:\Cc_\BB\to \Cc$ and the embedding functor $\Cc^f\hookrightarrow \Cc$, as in the following diagram.
\[
\xymatrix{
\Cc_\BB \ar[rr]^{U} && \Cc\\
\Cc_\BB^f \ar@{^(->}[u] \ar[rr]^{U^f} && \Cc^f \ar@{^(->}[u]
}
\]
Explicitly, $\Cc^f_\BB$ consists of all objects $(X,\rho)\in \Cc_\BB$ such that $U(X,\rho)=X$ is rigid. 
We denote $[X,-]$ for the right adjoint of $X\ot -:\Cc\to \Cc$ and $[(X,\rho),-]_\BB$ for the right adjoint of $(X,\rho)\ot -:\Cc_\BB\to \Cc_\BB$. Then we have,
$U\circ [(X,\rho),-]_\BB=[X,-]\simeq X^*\ot -$. Moreover, $[(X,\rho),I]_\BB\in\Cc_\BB$, so $X^*$ can be endowed with a $\BB$-module structure $\rho^*$, and we find that $(X,\rho)$ is right rigid with dual $(X^*,\rho^*)$. 
\end{proof}

As \thref{Tannakabimonad} and \thref{TannakaHopfmonad} are ``if and only if'' theorems, where \coref{TannakaHopfmonadif} only works in one direction, the natural question arises whether the statement of \coref{TannakaHopfmonadif} also has an inverse, that is: ``Can we reconstruct a Hopf monad structure on $\BB$ by only knowing the rigid monoidal structure of $\Cc^f_\BB$''. The first problem is that in general, the monad $\BB$ is not a functor defined on $\Cc^f$. If one thinks about (Hopf) monads that arise from ordinary $k$-(Hopf) algebras, this is only the case if the underlying space of the (Hopf) algebra is itself finite dimensional. In this case we can apply \cite[Theorem 3.10]{BLV} and find 

\begin{theorem}\thlabel{Tannakafinitedim}
Let $\Cc$ be a rigid monoidal category and $(\BB,M,U)$ a monad on $\Cc$. There is a bijective correspondence between
\begin{enumerate}[(i)]
\item Hopf monad structures on $\BB$ with underlying monad $(\BB,M,U)$; and
\item rigid monoidal structures on the Eilenberg-Moore category of the monad $(\BB,M,U)$ such that the forgetful functor to $\Cc$ is strict monoidal 
(and hence rigid).
\end{enumerate}
\end{theorem}

\subsection{Simple reconstruction of bialgebras and Hopf algebras}
\selabel{simplealgebra}

In this section, we will apply the results of the previous section on the case where the bimonad is obtained from a bialgebra. We give a reconstruction-type theorem for bialgebras and Hopf algebras that are not yet of the type of Tannaka reconstruction theorem. That is, we do not reconstruct a the whole algebra from only pieces of finite dimensional information, rather we reconstruct the coalgebraic structure on a given algebra from the monoidal structure on its category of representations. 
This is however the right step up towards a full Tannaka reconstruction theorem, as it fully characterizes bialgebras and Hopf algebras, as one sees from \thref{Reconstrbialg} and \thref{ReconstrHopfalg} respectively.

\begin{theorem}\thlabel{Reconstrbialg}
Let $B=(B,m_B,u_B)$ be an algebra in a braided monoidal category $\Aa=(\Aa,\ot,I,\gamma)$. Suppose that $I$ is a regular generator of $\Aa$ and that $-\ot X$ and $X\ot -$ preserve colimits in $\Aa$, for all $X\in \Aa$. 
There is a bijective correspondence between
\begin{enumerate}[(1)]
\item monoidal structures on $\Aa_B$ such that the forgetful functor
$\Aa_B \to \Aa$ is strict monoidal, and
\item bialgebra structures $(B,m_B,u_B,\Delta_B,\varepsilon_B)$ on $B$.
\end{enumerate}
\end{theorem}
\begin{proof}
Follows from \thref{pretannakabialgebra} and \thref{Tannakabimonad}
\end{proof}

\begin{theorem}\thlabel{ReconstrHopfalg}
Let $B=(B,m_B,u_B)$ be an algebra in a closed braided monoidal category $\Aa=(\Aa,\ot,I,\gamma)$. Suppose that $I$ is a regular generator of $\Aa$ and that $-\ot X$ and $X\ot -$ preserve colimits in $\Aa$, for all $X\in \Aa$. 
There is a bijective correspondence between
\begin{enumerate}[(1)]
\item right closed monoidal structures on $\Aa_B$ such that the forgetful functor
$\Aa_B \to \Aa$ is strict monoidal and right closed, and
\item Hopf algebra structures $(B,m_B,u_B,\Delta_B,\varepsilon_B)$ on $B$.
\end{enumerate}
\end{theorem}

\begin{proof}
Follows from \thref{pretannakabialgebra} and \thref{TannakaHopfmonad}.
\end{proof}

Remark that the last two Theorems are in particular applicable to the case $\Cc=\Mm_k$, i.e.\ they characterize  classical bialgebras and Hopf algebras. 

\subsection{Tannaka reconstruction of bialgebras and Hopf algebras}
\selabel{realTannaka}

We will not discuss the Tannaka reconstruction theorem in terms of Hopf monads or bimonads, of which an explicit description is not known to the author (although we can refer the interested reader to the recent \cite{BooStr} for an even more general approach, see also \seref{prospects}).
In recent years, several variations and generalizations of the theorem in different settings have been formulated, see e.g. \cite{JoyStr:Tannaka}, \cite{Ma:book}, \cite{McC:Tannaka}, \cite{Schap}, \cite{Schap2}, \cite{Sch:Tannaka} and \cite{Szl:Tannaka}. 

Our formulation 
is taken from \cite{McC:Tannaka}. Let $\Cc$ be a complete braided monoidal category such that all endofunctors of the form $-\ot X$ and $X\ot -$ for $X\in \Cc$ preserve limits. The category ${\sf strmon}\swarrow \Cc$ is defined as the category whose objects are pairs $(\Vv,F)$, where $\Vv$ is a monoidal category and $F:\Vv\to\Cc$ is a strict monoidal functor $F:\Vv\to \Cc$ such that $FX$ is right rigid for any object $X\in \Vv$. A morphism $H:(\Vv,F)\to (\Ww,G)$ in this category is a strict monoidal functor $H:\Vv\to \Ww$ such that $F=G\circ H$.
The category ${\sf strmon}^*\swarrow \Cc$ is the subcategory of ${\sf strmon}\swarrow \Cc$ that consists of functors $F:\Vv\to \Cc$, where $\Vv$ is a right rigid category. 

Let $\Cc$ be a braided monoidal category.
As we know from \thref{Reconstrbialg} every bialgebra $B$ in $\Cc$ induces a strict monoidal functor $U_B:\Cc_B\to \Cc$. If $B$ is moreover a Hopf algebra, then we have by \coref{TannakaHopfmonadif} also functor $U^f_B:\Cc_B^f\to \Cc^f$. Therefore, we find functors $\tilde U:\Bialg(\Cc)\to {\sf strmon}\swarrow \Cc,\ \tilde U(B)=(\Cc_B,U_B)$ and $\tilde U^*:\Hpfalg(\Cc)\to {\sf strmon}^*\swarrow \Cc,\ \tilde U^*(B)=(\Cc_B^f,U^f_B)$. The Tannaka reconstruction allows in first place a left adjoint for these functors.

\begin{theorem}\thlabel{tannakaadjunction}
Let $\Cc$ be a complete braided monoidal category such that all endofunctors of the form $-\ot X$ and $X\ot -$ for $X\in \Cc$ preserve limits.
\begin{enumerate}[(i)]
\item The functor $\tilde U:\Bialg(\Cc)\to {\sf strmon}\swarrow \Cc$ has a left adjoint ${\sf tan}$.\\
\item The functor $\tilde U^*:\Hpfalg(\Cc)\to {\sf strmon}^*\swarrow \Cc$ has a left adjoint ${\sf tan}^*$.\\
\end{enumerate}
\end{theorem}

The existence of the functors ${\sf tan}$ and ${\sf tan}^*$ is based on the so-called {\em end}-construction, which is in fact a particular limit, see e.g. \cite{Kelly}. This construction allows to build up an algebra out of it's category of representations, or dually a coalgebra out of it's category of corepresentations (in the latter case this coalgebra is sometimes called the {\em coendomorphism coalgebra}, or {\em coend} for short). As in our situation, the representation categories posses an additional monoidal structure, the reconstructed algebra will inherit an additional structure as well, leading to a bialgebra or Hopf algebra. We refer for a full proof of \thref{tannakaadjunction} to e.g.\ \cite[Section 16]{Str:QG} or \cite[Section 6]{McC:Tannaka}.

The next question that arises, is whether the above theorem completely determines bialgebras and Hopf algebras. That is, when the functors $\tilde U$, $\tilde U^*$ or their adjoints are fully faithful. In particular, if $H$ is a bialgebra (or a Hopf algebra), one can wonder if $H$ is isomorphic to the reconstructed algebra ${\sf tan} \tilde U H$, this is the so-called {\bf reconstruction problem}. 
A second problem is termed {\bf recognition problem} and refers to the fact whether the pair $\tilde U {\sf tan}(\Vv,F)$ is isomorphic to $(\Vv,F)$ in ${\sf strmon}\swarrow \Cc$, that is, whether the functor $F$ is essentially unique.

It turns out that in many of the cases of interest, such as when $\Cc=\vect(k)$, the category of vector spaces over a fixed field $k$, both problems have a positive answer, leading in particular to \thref{TannakaHopfalg}. Generalizations to a general categorical setting often become highly technical, and we omit them explicitly here. Let us briefly summarize some results.
\begin{itemize}
\item B.\ Day \cite{Day} solved both problems 
for finitely presentable, complete and cocomplete symmetric monoidal closed categories for which the full subcategory of objects with duals is closed under finite limits and colimits. 
\item P.\ McCrudden \cite{McC:Masckean}
proved the reconstruction problem for so-called Maschkean categories, which are 
certain abelian monoidal categories in which all monomorphisms split 
\item Probably the most general (symmetric) setting can be found in \cite{Schap}, where the author deals with complete and cocomplete symmetric monoidal and closed categories (called cosmoi). 
\end{itemize}

\section{Varations on the notion of Hopf algebra}\selabel{Variations}

As we have seen, the reconstruction theorems (\thref{TannakaHopfalg} and \thref{ReconstrHopfalg}) fully characterize Hopf algebras over a field. By varying the base category or the properties of the forgetful functor, we will recover in this section different variations on the notion of a classical Hopf algebra, that were defined over the last decades. 

\subsection{Variations on the properties of the forgetful functor}

\subsubsection{Quasi Hopf algebras}

Let $(\Cc,\ot,I)$ and $(\Dd,\odot,J)$ be monoidal categories, we will say that a functor $F:\Cc\to\Dd$ is a quasi-monoidal functor if there is a natural isomorphism $\psi_{X,Y}:F(X)\odot F(Y)\cong F(X\ot Y)$ and a $\Cc$-isomorphism $\psi_0:F(I)\cong J$ (without any further conditions). Clearly any strong monoidal functor is a quasi-monoidal functor. 
Then we can introduce quasi bialgebras by postulating the following characterization:
\begin{quote}
Let $(H,m,u)$ be a $k$-algebra. Then there is a bijective correspondence between
\begin{enumerate}[(i)]
\item {\em quasi bialgebra} structures on $H$ with underlying algebra $(H,m,u)$; and
\item monoidal structures on the category of right $H$-modules $\Mm_H$ such that the forgetful functor $U:\Mm_H\to \Mm_k$ is a quasi monoidal functor.
\end{enumerate}
\end{quote}

Clearly, every (usual) $k$-bialgebra is a quasi bialgebra, but the converse is not true. The main difference with usual bialgebras is that the comultiplication in a quasi bialgebra is not necessarily coassociative. Quasi bialgebras were introduced by Drinfel'd in 1989 \cite{Drin:QH} using the following more explicit description. Let $H$ be a $k$-algebra with an invertible element $\Phi\in H\ot H\ot H$, and endowed with a comultiplication $\Delta:H\to H\ot H$ and a counit $\epsilon: H\to k$ satisfing the following conditions for all $a \in H$
\begin{eqnarray*}
    (H \otimes \Delta) \circ \Delta(a) &=& \Phi \lbrack (\Delta \otimes H) \circ \Delta (a) \rbrack \Phi^{-1},  \\
    (\varepsilon \otimes H) \circ \Delta = &H& = (H \otimes \varepsilon) \circ \Delta.
\end{eqnarray*}
Furthermore, $\Phi$  has to be a normalized 3-cocycle, in the sense that
\begin{eqnarray*}
    [ (H \otimes H \otimes \Delta)(\Phi) ] \ [ (\Delta \otimes H \otimes H)(\Phi) ] 
    &=& (1 \otimes \Phi) \ \lbrack (H \otimes \Delta \otimes H)(\Phi) \rbrack \ (\Phi \otimes 1) \\
    (H \otimes \varepsilon \otimes H)(\Phi) &=& 1 \otimes 1. 
\end{eqnarray*}
Then $H$ is a quasi bialgebra. The correspondence with the characterization above, follows by the fact that the associativity constraint in the monoidal category of right $H$-modules $\Mm_H$ over a quasi bialgebra $H$ can be constructed as
$$a_{X,Y,Z}:X\ot (Y\ot Z)\to (X\ot Y)\ot Z,\ a_{X,Y,Z}(x\ot (y\ot z))=((x\ot y)\ot z)\cdot \Phi$$
Conversely, if $\Mm_H$ is monoidal, then we recover $\Phi=\alpha_{H,H,H}(1\ot 1\ot 1)$, which satisfies the conditions of a normalized $3$-cocycle as a consequence of the constraints in a monoidal category.

If moreover, there exist elements $\alpha,\beta\in H$ and an anti-algebra morphism $S:H\to H$ such that 
\begin{eqnarray*}
S(a_{(1)})\alpha a_{(2)} = \epsilon(a)\alpha; && a_{(1)}\beta S(a_{(2)})=\epsilon(a)\beta,
\end{eqnarray*}
for all $a\in H$ and
\begin{eqnarray*}
X^1 \beta S(X^1) \alpha X^3 = &1& =S(x^1) \alpha x^2 \beta S(x^3). 
\end{eqnarray*}
where we have denoted $\Phi = X^1 \otimes X^2 \otimes X^3$ and $\Phi^{-1}= x^1 \otimes x^2 \otimes x^3$, 
then $H$ is called a quasi Hopf algebra. By means of Tannaka reconstruction, one then finds the following characterization.

\begin{quote}
\index{Hopf algebra!quasi-}
There is a bijective correspondence between:
\begin{enumerate}[(i)]
\item quasi Hopf $k$-Hopf algebras $H$; and
\item right rigid monoidal categories $\Vv$ together with a right rigid quasi monoidal functor $\Vv\to \Mm_k$.
\end{enumerate}
\end{quote}

Usual group algebras give rise to usual Hopf algebras, similarly examples of quasi Hopf algebras can be constructed by deforming group algebras with normalized 3-cocycles on this group in classical sense. As quasi Hopf algebras are in bijective correspondence with certain classes of monoidal categories, these results can be used to compute all possible monoidal structures on certain (small) categories, this is for example done in \cite{BulCaeTor:Klein}, see also references therein.

In contrast to usual bialgebras and Hopf algebras, which are self-dual objects in a braided monoidal category, the quasi-version is not self-dual. Nevertheless, it is possible to define ``dual quasi bialgebras'' or ``co-quasi bialgebras'' (and Hopf algebras). In this dual setting, the objects are usual coalgebras, but posses a non-associative algebra structure, that is governed by a ``reassociator'' $\phi\in\JVHom(H\ot H\ot H,k)$. 
More precisely, a co-quasi bialgebra is a coalgebra $H$ such that its category of comodules $\Mm^H$ is monoidal and the forgetful functor to $k$-modules is quasi monoidal.
Hence one can consider algebras in the category $\Mm^H$, called $H$-comodule algebras. 
Let $A$ be an $H$-comodule algebra. 
Then $A$ is a right $H$-comodule equipped with $H$-comodule morphisms $\mu:A\ot A\to A, \mu(a\ot b)=a\cdot b$ and $\eta: k\to A$. However, the triple $(A,\mu,\eta)$ is not an associative $k$-algebra. In contrast, $A$ satisfies the following quasi-associativity condition
$$(a\cdot b)\cdot c=a_{[0]}\cdot (b_{[0]}\cdot c_{[0]})\phi(a_{[1]},b_{[1]},c_{[1]}).$$
It should be remarked that different from the classical case, $H$ with regular multiplication is not an $H$-comodule algebra. Since the forgetful functor $U:\Mm^H\to \Mm_k$ is a not a strong monoidal functor, 
the underlying $k$-vectorspace of the $H$-comodule algebra $A$ does in general no longer possess the structure of an associative $k$-algebra.

Recall that two co-quasi bialgebras $H$ and $H'$ are called gauge-equivalent iff there exists a monoidal equivalence $F:\Mm^H\to\Mm^{H'}$ that commutes with the forgetful functors. In particular, if
a co-quasi Hopf algebra $H$ is gauge equivalent with a usual Hopf algebra $H_F$, this means that we have a monoidal functor $F:\Mm^H\to \Mm^{H_F}$ such that $U_H=U_{H_F}\circ F:\Mm^H\to \Mm_k$.
As we know that the forgetful functor $U_{H_F}:\Mm^{H_F}\to \Mm_k$ is strict monoidal, we find that the forgetful functor $U_H$ is again monoidal, although not necessarily strict monoidal.
Consequently, we obtain that any (initial non-associative) $H$-comodule algebra, also possesses the structure of an associate $k$-algebra, by deforming its multiplication by means of the functor $U_H$ in the following way 
$$\xymatrix{\mu':A\ot A=U_HA\ot U_HA \ar[rr] && U_H(A\ot A) \ar[rr]^-{\mu} && U_HA=A},$$
where $\mu$ is the multiplication of the $H$-comodule algebra $A$.
Using the converse argumentation, certain associative $k$-algebras can be deformed into non-associative ones by means of a gauge-transform. This idea is explored in a very elegant way in \cite{AlbMa}, where it is shown that the octonions arise as a deformation of the group algebra $k[\ZZ_2\times \ZZ_2\times \ZZ_2]$ by a $2$-cochain, and in this way can be interpreted as a comodule algebra over a co-quasi Hopf algebra.

\subsubsection{Quasitriangular Hopf algebras}

Similar as in the previous section, we can introduce quasitriangular Hopf algebras by postulating the following characterization:
\begin{quote}
\index{Hopf algebra!quasitriangular}
Let $(H,m,u)$ be a $k$-algebra. Then there is a bijective correspondence between
\begin{enumerate}[(i)]
\item {\em quasitriangular Hopf algebra} structures on $H$ with underlying algebra $(H,m,u)$; and
\item right closed braided monoidal structures on the category of right $H$-modules $\Mm_H$ such that the forgetful functor $U:\Mm_H\to \Mm_k$ is a right closed braided strict monoidal functor.
\end{enumerate}
\end{quote}

Quasitriangularity is a in fact a property of a Hopf algebra, not a true variation on the axioms. 
Every quasitriangular Hopf algebra is a Hopf algebra, but not conversely. Explicitly, a Hopf algebra $(H,m,u,\Delta,\epsilon,S)$ is quasitriangular if there exists an invertible element (called the $R$-matrix) $R=r^1\ot r^2=R^1\ot R^2\in H \otimes H$ such that  
\begin{eqnarray*}
r^1x_{(1)}\ot r^1x_{(2)} &=& x_{(2)}r^1\ot x_{(1)}r^2,\\
r^1_{(1)}\ot r^1_{(2)}\ot r^2 &=& r^1\ot R^1\ot r^2R^2, \\
r^1\ot r^2_{(1)}\ot r^2_{(2)} &=& r^1R^1\ot R^2\ot r^1;
\end{eqnarray*}
for all $x \in H$.

As a consequence of the properties of quasitriangularity, the $R$-matrix is a solution of the Yang-Baxter equation. Therefore $H$-modules over a quasitriangular Hopf algebra are for example studied to determine quasi-invariants of braids and knots. In fact, quasitriangularity can already be considered for bialgebras, leading to the expected versions of reconstruction and Tannaka theorems.

\subsubsection{Weak Hopf algebras}

\index{Functor!Frobenius monoidal}
Let $(\Cc,\ot,I)$ and $(\Dd,\odot,J)$ be monoidal categories. A functor $F:\Cc\to\Dd$ is called {\em Frobenius monoidal} if $F$ has a monoidal structure $(\phi,\phi_0)$ and op-monoidal structure $(\psi,\psi_0)$ such that the following diagrams commute for all $A,B,C\in \Cc$
\begin{eqnarray*}
&\xymatrix{
F(A\ot B)\odot FC \ar[rr]^-{\phi_{A\ot B,C}} \ar[d]_-{\psi_{A,B}\odot FC} && F(A\ot B\ot C) \ar[d]^-{\psi_{A,B\ot C}}\\
FA \odot FB \odot FC \ar[rr]_-{FA\odot \phi_{B,C}} && FA\odot F(B\ot C)
}\\
&\xymatrix{
FA\odot F(B\ot C) \ar[rr]^-{\phi_{A,B\ot C}} \ar[d]_-{FA\odot\psi_{B,C}} && F(A\ot B\ot C) \ar[d]^-{\psi_{A\ot B, C}}\\
FA \odot FB \odot FC \ar[rr]_-{\phi_{A,B}\odot FC} && F(A\ot B)\odot FC
}
\end{eqnarray*}
A Frobenius monoidal functor is called {\em separable Frobenius monoidal} if moreover 
$$\phi_{A,B}\circ \psi_{A,B}=F(A\ot B)$$
for all $A,B\in\Cc$. Any strong monoidal functor is separable Frobenius monoidal.

Again, we introduce the next notion by postulating the following characterization
\begin{quote}
Let $(H,m,u)$ be a $k$-algebra. Then there is a bijective correspondence between
\begin{enumerate}[(i)]
\item {\em weak bialgebra} structures on $H$ with underlying algebra $(H,m,u)$; and
\item monoidal structures on the category of right $H$-modules $\Mm_H$ such that the forgetful functor $U:\Mm_H\to \Mm_k$ is a separable Frobenius monoidal functor.
\end{enumerate}
\end{quote}

The classical definition of a weak bialgebra and weak Hopf algebra was given in \cite{bohm:weakhopf}.
A $k$-algebra $(H,m,u)$ is a weak bialgebra if $H$ has a $k$-coalgebra structure $(H,\Delta,\epsilon)$, such that $\Delta$ is a multiplicative map (i.e. the first diagram of \equref{bialgebra} commutes) and the following weaker compatibility conditions hold
\begin{eqnarray*}
&(\Delta(1)\ot 1)(1\ot \Delta(1))=(\Delta\ot H)\Delta(1)=(1\ot \Delta(1))(\Delta(1)\ot 1)\\
&\epsilon(b1_{(1)})\epsilon(1_{(2)}b')=\epsilon(bb')=\epsilon(b1_{(2)})\epsilon(1_{(1)}b'),
\end{eqnarray*}
for $b,b'\in H$. A weak Hopf algebra is a weak bialgebra that is equipped with a $k$-linear map $S:H\to H$ satisfying
\begin{eqnarray*}
h_{(1)}S(h_{(2)})&=&\epsilon(1_{(1)}h)1_{(2)};\\
S(h_{(1)})h_{(2)}&=&1_{(1)}\epsilon(h1_{(2)});\\
S(h_{(1)})h_{(2)}S(h_{(3)})&=&S(h).
\end{eqnarray*}
The Tannaka reconstruction gives us the following characterization (see e.g.\ \cite{McC:Tannaka}).

\begin{quote}
There is a bijective correspondence between:
\begin{enumerate}[(i)]
\item weak Hopf $k$-Hopf algebras $H$; and
\item right rigid monoidal categories $\Vv$ together with a right rigid separable Frobenius monoidal functor $\Vv\to \Mm_k$.

\end{enumerate}
\end{quote}

Weak Hopf algebras are in relation with bimonads and Hopf monads, as we will discuss in more detail in \seref{bialgebroids}. To study the particularities of weak Hopf algebra theory however, weak monads, weak bimonads and weak Hopf monads were introduced in a series of papers (see e.g.\ \cite{bohm:weakmonad} \cite{BLS}). 

\subsection{Variations on the monoidal base category}

\subsubsection{Hopf algebroids}\selabel{bialgebroids}

It took quite a long time to establish the correct Hopf-algebraic notion over a non-commutative base. The reason of the difficulties are quite clear. First of all, if $R$ is a non-commutative ring, then the category of right $R$-modules $\Mm_R$ is no longer monoidal (in general). Therefore, we have to look in stead to the category of $R$-{\em bimodules} ${_R\Mm_R}$, which is monoidal, but in general still not braided. So bialgebras and Hopf algebras can not be computed {\em inside} this category. However, we can compute bimonads and Hopf-monads {\em on} this category. This is how the theory of bialgebroids can be developed. It has to be told that historically, bialgeboids and Hopf algebroids were constructed first in a more direct way, and the interpretion via bimonads and Hopf monads is only very recent. 
However, in order to make the relation between the different variations on Hopf algebra-like structures more prominent, we take this converted approach in this note and introduce bialgebroids by postulating the following characterization (reformulation of \cite{Schau}[Theorem 5.1]):
\begin{quote}
\index{Bialgebroid}
Let $B$ be an $R\ot R^{op}$-algebra $(B,m,u)$ (i.e.\ an algebra in the monoidal category ${_R\Mm_R}$), then there is a bijective correspondence between
\begin{enumerate}[(i)]
\item (right) {\em $R$-bialgebroid} structures on $B$, with underlying algebra $(B,m,u)$; and
\item monoidal structures on the category of right $B$-modules such that the forgetful functor $U:\Mm_B\to {_R\Mm_R}$ is strict monoidal. 
\end{enumerate}
\end{quote}
A particular feature of bialgebroids, is that rather than a unit map $u:R\ot R^{op}\to B$, one considers the source and target maps $s:R\to B$ and $t:R^{op}\to B$, which are the combination of the unit map $u$ with the canonical injections $R\to R\ot R^{op}$ and $R^{op}\to R\ot R^{op}$ (respectively).

As one can see, different from the case over a commutative base, the notion of a bialgeboid is not left-right symmetric. A left bialgebroid is introduced symmetrically, as an $R\ot R^{op}$-algebra with monoidal structure on its category of left modules.

Due to this assymetry, several different notions of a Hopf algebroid were introduced in the literature. Some of these were shown to be equivalent, although this was far from being trivial. We omit this disscusion here, but refer to the survey \cite{bohm:HoA}. The presently overall accepted notion of a Hopf algebroid (introduced in \cite{BohmSzl:hgdax} for bijective antipodes) consists of a triple $(H_L,H_R,S)$, where $H_L$ is a left $L$-algebroid, $H_R$ is a right $R$-algebroid, such that $H_L$ and $H_R$ share the same underlying $k$-algebra $H$. The structure maps have to satisfy several compatibility conditions for which we refer to 
\cite{bohm:HoA}[Definition 4.1], and $S:H\to H$ is the $k$-linear antipode map that satisfies the following axiom
$$\mu_L\circ (S\ot_L H)\circ \Delta_L=s_R\circ \epsilon_R,\qquad \mu_R\circ (H\ot_R S)\circ \Delta_R=s_L\circ \epsilon_L,$$
where $\mu_L, \Delta_L, \epsilon_L$ and $s_L$ are respectively the multiplication map, the comultiplication map, the counit map and the source map of the left bialgebroid $H_L$.

The way we defined a right bialgebroid $B$ over $R$, tells immediately that $-\ot_RB$ is a bimonad on ${_R\Mm_R}$.
It turns out that the bimonad $-\ot_RH_R$ associated to a Hopf algebroid is a right Hopf monad on ${_R\Mm_R}$ and the bimonad $H_L\ot_L-$ becomes a left Hopf monad on ${_L\Mm_L}$. The converse is not always true, but the alternative notion of a $\times_R$-Hopf algebra introduced by Schauenburg (see \cite{Schau:xRHopf}) is defined as a right $R$-bialgebroid such that the canonical morphism 
$\can:H\ot_{R^{op}}H\to H\ot_RH,\ \can(h\ot_{R^{op}}h')=h_{(1)}\ot_Rh_{(2)}h'$ is a bijection.
This leads to the following  characterization.
\begin{quote}
Let $B$ be an $R\ot R^{op}$-algebra, then there is a bijective correspondence between
\begin{enumerate}[(i)]
\item $\times_R$-Hopf algebra structures on $B$,
\item right closed monoidal structures on the category of right $B$-modules such that the forgetful functor $U:\Mm_B\to {_R\Mm_R}$ is right closed.
\end{enumerate}
\end{quote}

It should be remarked that weak Hopf algebras (respectively weak bialgebras) are strongly related to Hopf algebroids (resp. bialgebroids). To some extend, weak Hopf algebras motivated largely the recent development in the theory of Hopf algebroids. 
Let $H$ be a weak bialgebra. Then $\Mm_H$ is a monoidal category, but the forgetful functor $\Mm_H\to \Mm_k$ is not strict monoidal, not even strong monoidal, it is only separable Frobenius monoidal. However, if we consider the $H$-subalgebra
$$R=\JVim \pi_R, \quad {\rm where}\quad \pi_R:H\to H,\ \pi_R(b)=1_{(1)}\epsilon(b1_{(2)}),$$
which is called the target space, then we do obtain a strict monoidal functor $\Mm_H\to {_R\Mm_R}$. 
In this way, the weak bialgebra $H$ becomes a right $R$-bialgebroid, and in a similar way a weak Hopf algebra becomes a weak Hopf algebroid see \cite[section 3.2.2 and section 4.1.2]{bohm:HoA}. For a detailed discussion of the monoidal properties of the representation categories of weak bialgebras, we refer to \cite{BCJ}.

\subsubsection{Hopf Group coalgebras}

Hopf group-coalgebras were introduced by Turaev in his work on homotopy quantum field theories (see \cite{Tur2} and the earlier preprint \cite{Tur}). The purely algebraic study of these objects was initiated by Virelizier in \cite{Vir}. Explicitly, a Hopf group-coalgebra is a family of algebras $(H_g,\mu_g,\eta_g)_{g\in G}$ indexed by a group $G$ with unit $e$, together with a family of algebra maps
$$\Delta_{g,h}:H_{gh}\to H_g\ot H_h,\quad \forall g,h\in G$$
and an algebra map $\epsilon:H_e\to k$ and a family of $k$-linear maps $S_g:H_{g^{-1}}\to H_g,\ \forall g\in G$ such the following compatibility conditions hold for all $g,h,f\in G$
\begin{eqnarray*}
&(\Delta_{g,h}\ot H_f)\circ \Delta_{gh,f}=(H_g\ot \Delta_{h,f})\circ \Delta_{g,hf}\\
&(H_g\ot \epsilon)\circ \Delta_{g,e}=H_g=(\epsilon\ot H_g)\circ\Delta_{e,g}\\
&\mu_g\circ (S_g\ot H_g)\circ \Delta_{g^{-1},g}=\eta_g\circ \epsilon = \mu_g\circ (H_g\ot S_g)\circ \Delta_{g,g^{-1}}
\end{eqnarray*}
In \cite{CDL} the nice observation was made that these objects can be understood as Hopf algebras in a particular symmetric monoidal category.

Let us first recall a general construction in category theory.
Let $\Cc$ be a category. Then we can construct a new category $\Fam(\Cc)$, the category of families in $\Cc$, as follows:
\begin{itemize}
\item an object in $\Fam(\Cc)$ is a pair $(I,\{C_i\}_{i\in I})$, where $I\in\Set$;
\item A morphism in $\Fam(\Cc)$ is a pair $(f,\phi):(I,\{C_i\}_{i\in I})\to (J,\{D_j\}_{j\in J})$ consisting of a map $f:I\to J$ and a family of $\Cc$-morphisms $\phi_i:C_i\to D_{f(i)}$.
\end{itemize} 
Dually, we define $\Maf(\Cc)=\Fam(\Cc^{op})^{op}$. This is the category with the same objects, but morphisms are pairs $(f,\phi):(I,\{C_i\}_{i\in I})\to (J,\{D_j\}_{j\in J})$ consisting of a map $f:J\to I$ and a family of $\Cc$-morphisms $\phi_j:C_{f(j)}\to D_{j}$. If $\Cc$ is monoidal, braided monoidal or closed than $\Fam(\Cc)$ and $\Maf(\Cc)$ are as well in a canonical way. In \cite{CDL}, the category $\Fam(\Cc)$ was called the Zunino category and $\Maf(\Cc)$ was called the Turaev category associated to $\Cc$. The reason for these names is the observation we borrow from the same paper, based on the computation of algebras, coalgebras, bialgebras and Hopf algebras in these categories.
\begin{itemize}
\item Algebras in $\Fam(\Cc)$ are nothing else than algebras graded by a monoid; coalgebras in $\Fam(\Cc)$ are just a family of coalgebras indexed by a set.
\item Algebras in $\Maf(\Cc)$  are nothing else than families of algebras indexed by a set; coalgebras in $\Maf(\Cc)$ are coalgebras that are a kind of co-graded coalgebras, called $G$-coalgebras in \cite{Tur}.
\item bialgebras and Hopf algebras in $\Fam$ are Hopf algebras graded by a group and bialgebras and Hopf algebras in $\Maf$ are ``co-graded'' versions, called group Hopf coalgebras in \cite{Tur}.
\end{itemize}
As Hopf group coalgebras are just Hopf algebras a particular braided monoidal category, the reconstruction theorems \thref{Reconstrbialg} and \thref{ReconstrHopfalg} can be directly applied to this situation.

\subsubsection{Multiplier Hopf algebras}
 
Let $A$ be a non-unital algebra. An (right) $A$-module is called {\em firm} as if the multiplication map induces an isomorphism $M\ot_AA\cong M$. The algebra $A$ is called a firm algebra if it is firm as left or equivalently right regular $A$-module.
In this situation, then the category of firm  $A$-bimodules is again a monoidal category with monoidal unit $A$. 
Examples of this kind of non-unital algebras are so-called algebras with {\em local units}, that is, algebras such that for any $a\in A$, there exists an element $e\in A$ such that $ae=a=ea$. If $S$ is an infinite set, then the algebra of functions with finite support $\fHom(S,k)$ is a non-unital algebra with local units. 

{Multiplier Hopf algebras} are a generalization of Hopf algebras in the setting of non-unital algebras. They were motivated by study of non-compact quantum groups in the setting $C^*$-algebras, but studied in a purely algebraic setting since introduction. The non-compactness of their underlying space is directly related to the fact that the algebra is non-unital. However, it was proven that multiplier Hopf algebras always have local units. 

If $A$ is a non-unital $k$-algebra, then the multiplier algebra of $A$ is the $k$-module $M(A)$ that is defined by the following pullback
\[
\xymatrix{
M(A)\ar[rr] \ar[d] && {_A\JVEnd}(A) \ar[d]^{\ol{(-)}}\\
\JVEnd_A(A) \ar[rr]_-{\ul{(-)}} && {_A\JVHom_A}(A\ot A,A)
}
\]
where we used the linear maps
\begin{eqnarray}
\ol{(-)}:{_A\JVEnd}(A)\to {_A\JVHom_A}(A\ot A,A),\ \bar{\rho}(a\ot b)=\rho(a)b,\ {\rm for}\ \rho\in {_A\JVEnd}(A);\ \ \\
\ul{(-)}:\JVEnd_A(A)\to {_A\JVHom_A}(A\ot A,A),\ \ul{\lambda}(a\ot b)=a\lambda(b),\ {\rm for}\ \lambda\in \JVEnd_A(A).\ \
\end{eqnarray}
Remark that if $A$ is unital then $A\cong \JVEnd_A(A)\cong {_A\JVEnd}(A)\cong {_A\JVHom_A}(A\ot A,A)$ in a canonical way, 
hence also $M(A)\cong A$.
We can understand $M(A)$ as the set of pairs $(\lambda,\rho)$, where $\lambda\in \JVEnd_A(A)$ and $\rho\in {_A\JVEnd}(A)$, such that 
\begin{equation}\eqlabel{compatlambdarho}
a\lambda(b)=\rho(a)b,
\end{equation}
for all $a,b\in A$. Elements of $M(A)$ are called \emph{multipliers}. For any $x\in M(A)$, we will represent this element as $(\lambda_x,\rho_x)$. Moreover, for any $a\in A$, we will denote
$$a\cdot x=\rho_x(a), \quad x\cdot a=\lambda_x (a),$$
then \equref{compatlambdarho} reads as $a(xb)=(ax)b$. 
One obtains in a canonical way that $M(A)$ is a unital algebra and there is a canonical algebra morphism $\iota:A\to M(A)$. Moreover, if $f:A\to M(B)$ is a morphism of algebras and the maps
\begin{eqnarray*}
A\ot B \to B,\ a\ot b\mapsto f(a)\cdot b\\
B\ot A\to B,\ b\ot a\mapsto b\cdot f(a)
\end{eqnarray*}
are surjective, one can always extend this map to a morphism $\bar f:M(A)\to M(B)$. Such morphisms are called non-degenerate.

Let $A$ be an algebra with local units. Consider a non-degenerate algebra morphism $\Delta:A\to M(A\ot A)$, such that for all $a,b\in A$, $\Delta(a)(1\ot b)\in A\ot A$ and $(b\ot 1)\Delta(a)\in A\ot A$. Then we can express the following coassociativity condition for all $a,b,c \in A$,
$$(c\ot 1\ot 1)(\Delta\ot A)(\Delta(a)(1\ot b)) = (A\ot \Delta)((c\ot 1)\Delta(a)) (1\ot 1\ot b).$$
Now consider the following ``fusion maps'' or canonical maps
\begin{eqnarray*}
  T_1:A\ot A\to A\ot A,&& T_1(a\ot b)=\Delta(a)(1\ot b);\\
T_2:A\ot A\to A\ot A,&& T_1(a\ot b)=(a\ot 1)\Delta(b).
\end{eqnarray*}
Following Van Daele \cite{VD:Mult}, we say that $A$ is a multiplier Hopf algebra if there is a non-degenerate coassociative comultiplication $\Delta:A\to M(A\ot A)$ as above such that the maps $T_1$ and $T_2$ are bijective.

It can be shown that $A$ is a multiplier Hopf algebra if and only if there exists a counit $\epsilon:A\to k$ and an antipode $S:A\to M(A)$ satisfying conditions similar to the classical case, but which have to be formulated with the needed care.

The full categorical description of multiplier Hopf algebras is not settled yet. A first attempt was made in \cite{JanVer}. In this paper a reconstruction theorem for multiplier bialgebras was given. By a multiplier bialgebra we mean a non-unital algebra $A$ (with local units), that has a coassociative non-degenerate comultiplication $\Delta:A\to M(A\ot A)$ and a counit $\epsilon:A\to k$. Then according \cite[Theorem 2.9]{JanVer}, we have the following characterization
\begin{quote}
Given an algebra with local units $A$, there is a bijective correspondence between 
\begin{enumerate}[(i)]
\item
multiplier bialgebra structures on $A$; and
\item 
monoidal structures on the categories $\Mm_A$ of firm right $A$-modules, ${_A\Mm}$ of firm left $A$-modules and the category $A-\JVExt$ of ring extensions $A\to A'$, where $A'$ is again a ring with local units, such that the following diagram is a diagram of forgetful functors is strict monoidal
\[
\xymatrix{
& A-\JVExt \ar[dr] \ar[dl] \\
{_A\Mm} \ar[dr] && \Mm_A \ar[dl]\\
& \Mm_k
}
\]
\end{enumerate}
\end{quote}

In forthcomming \cite{JanVer:Kleisli}, some classes of Multiplier Hopf algebras, such as multiplier Hopf algebras with a complete set of central local units including discrete quantum groups, are studied as Hopf algebras in a particular monoidal category. This monoidal category is closely related to the category $\Maf(\Cc)$ of the previous section, which indicates as well the close relationship between Hopf group coalgebras and multiplier Hopf algebras. In particular, for a group Hopf coalgebra $(H_g)_{g\in G}$, we have that $\oplus_{g\in G} H_g$ is a multiplier Hopf algebra, whose multiplier algebra is given by $M(\oplus_{g\in G} H_g)=\prod_{g\in G}H_g$.

\subsubsection{Hom-Hopf algebras}
Another variation of the classical notion of Hopf algebra that got a lot of attention recently, are so-called Hom-Hopf algebras (see \cite{HomHopf}). This concerns non-associative algebras $H$, whose non-associativity is ruled by a $k$-linear endomorphism $\alpha\in\JVEnd(H)$. It was proven in \cite{Isar} that Hom-Hopf algebras can be viewed as Hopf algebras in a monoidal category whose objects are pairs $(X,f)$, where $X$ is a $k$-module and $f$ is a $k$-automorphism of $X$, and where the associativity constraint is non-trivial. Moreover, an Hom-Hopf algebra is nothing else than a usual Hopf algebra, together with a Hopf algebra automorphism.

\subsection{More Hopf algebra-type structures and applications}

\subsubsection{Yetter-Drinfel'd modules}

For any monoidal category $\Cc$, one can construct its {\em center}, which is the braided monoidal category $\Zz(\Cc)$ whose objects are pairs $(A,u)$, where $A$ is an object of $\Cc$ and $u_X:A\ot X\to X\ot A$ is a natural transformation that satisfies
$$u_{X\ot Y}=(X\ot u_Y)\circ (u_X\ot Y),\quad u_I\simeq A$$
An arrow from $(A,u)$ to $(B,v)$ in $\mathcal{Z(C)}$ consists of an arrow $f:A \rightarrow B$ in $\mathcal{C}$ such that
$$v_X (f \otimes 1_X) = (1_X \otimes f) u_X.$$
The category $\mathcal{Z(C)}$ becomes a braided monoidal category with the tensor product on objects defined as
$$(A,u) \otimes (B,v) = (A \otimes B,w)$$
where $w_X = (u_X \otimes 1)(1 \otimes v_X)$, and the obvious braiding.
Let $H$ be a $k$-bialgebra. As we have seen, the category of right $H$-modules is a monoidal category. 
We can now define the category of Yetter-Drinfel'd modules as the center of  the monoidal category $H$-modules,
$\Yy\Dd^H_H=\Zz(\Mm_H)$. Moreover, as the tensor product is preserved, we obtain the following diagram of monoidal forgetful functors
\[
\xymatrix{
\Yy\Dd^H_H\ar[rr]\ar[rd] && \Mm_H \ar[dl] \\
&\Mm_k
}
\]
Hence there exists a quasi-triangular bialgebra $\Dd(H)$ such that $\Yy\Dd^H_H\cong \Mm_{\Dd(H)}$. We call $\Dd(H)$ the Drinfel'd double of $H$. 

Moreover, if $H$ is a finite dimensional Hopf algebra, then one can consider the rigid category $\Mm^f_H$ and it can be checked that $\Zz(\Mm^f_H)$ is rigid as well. This allows to construct a Hopf algebra structure on $\Dd(H)$, which can in this case be computed  explicitly as a crossed product on $H$ and $H^*$.

Furthermore, as $\Yy\Dd^H_H$ is again braided monoidal category, one can consider bialgebras and Hopf algebras {\em inside} this category. These objects are called Yetter-Drinfeld Hopf algebras and studied for example in \cite{Som}. As explained by Takeuchi \cite{Tak:braidedhopf} this is (almost) equivalent as to consider braided Hopf algebras, that is Hopf algebras in the category of $k$-modules, but where one uses a ``local braiding'' in stead of the usual twist maps.

\subsubsection{Monoidal structures on categories of relative Hopf modules}

Let $H$ be a bialgebra. Then we know that the category of (right) $H$-modules $\Mm_H$ and the category of (right) $H$-comodules $\Mm^H$ are monoidal categories. Hence we can consider bimonads on these categories, and Hopf monads on these categories if $H$ is a Hopf algebra. Suppose for example that $A$ is an $H$-comodule algebra. Than we have a monad 
$$\AA=-\ot A:\Mm^H\to \Mm^H.$$
Moreover, the category of $\AA$-modules coincides with the category of relative Hopf modules $\Mm^H_A$. By \thref{Tannakabimonad} $\AA$ is a bimonad if and only if $\Mm^H_A$ is a monoidal category and the forgetful functor $\Mm^H_A\to\Mm^H$ is a strict monoidal forgetful functor. 
For example, if $A$ is an bialgebra in the monoidal centre of $\Mm^H$, that is $A$ is a bialgebra in the category of Yetter-Drinfel'd modules over $H$, then the conditions are fulfilled and $\Mm^H_A$ is a monoidal category (see also \cite{BulCae}).

\subsubsection{Hopfish algebras}

Hopfish algebras were introduced in \cite{TWZ:Hopfish} and motivated by Poisson geometry, where there was a need to describe the structure on irrational rotation algebras, see \cite{BloTanWei}. From the purely algebraic point of view, a Hopfish algebras can be motivated by the very nice feature that they provide a Morita invariant notion of Hopf algebra-like structure. It is clear that the notion of a usual Hopf algebra is not at all Morita invariant: if an algebra $A$ is Morita equivalent with a Hopf algebra $H$, then $A$ has not necessarily the structure of a Hopf algebra. It follows from the theory of Tang, Weinstein and Zhu that $A$ however possesses a structure that they call a {\em Hopfish algebra}. Different from usual Hopf algebras, Hopfish algebras should be understood using  higher category theory. As usual Hopf algebras live in a braided monoidal base category, Hopfish algebras live in a monoidal bicategory. The theory for Hopfish algebras, and Tannaka theory in particular, is far from being fully explored. 

Suppose that $A$ is a Morita equivalent algebra with the Hopf algebra $H$. Then there is a strict Morita context $(A,H,D,E,\mu,\tau)$ and we have functors
$$\xymatrix{\Mm_A \ar[rr]^-F && \Mm_H \ar[rr]^-U && \Mm_k}.$$
Here $F\simeq -\ot_AD$ is an equivalence of categories, 
and $U$ is the strict monoidal forgetful functor. Using the equivalence between $\Mm_H$ and $\Mm_A$ we can put a monoidal structure on the category $\Mm_A$ by 
$$X\odot Y:= ((X\ot_A D)\ot (Y\ot_A D))\ot_H E$$
In particular, we see that the forgetful functor $U':\Mm_A\to \Mm_k$ is no longer a (strict) monoidal functor, which explains that $A$ is no longer a Hopf algebra, not even a bialgebra. The "comultiplication" is now a bimodule rather than a morphism. Indeed, if we calculate $A\odot A=(D\ot D)\ot_H E=:{\bf \Delta_A}$
Then this is naturally a left $A\ot A$-module and a right $A$-module. On the other side ${\bf \epsilon_A}:=k\ot_H E$ is a $k$-$A$ bimodule. The triple $(A,{\bf \Delta_A},{\bf \epsilon_A})$ then becomes a {\em sesquiunital sesquialgebra}, that satisfies the following coassociativity and counitality axioms:
$$(A\ot {\bf \Delta_A})\ot_{A\ot A} {\bf \Delta_A}\cong ({\bf \Delta_A}\ot A)\ot_{A\ot A}{\bf \Delta_A},$$
as $A\ot A\ot A$-$A$ bimodules and 
$$({\bf \epsilon_A}\ot A)\ot_{A\ot A}{\bf \Delta_A}\cong A\cong (A\ot {\bf \epsilon_A})\ot_{A\ot A}{\bf \Delta_A}$$
as $A$-bimodules. Conversely, the category of right $A$-modules of a sesquiunital sesquialgebra can be endowed with a monoidal structure, however without having a strict monoidal forgetful functor to the underlying category of $k$-modules. 

Furthermore, $A$ is a Hopfish algebra if it has antipode, which is defined as a left $A\ot A$-module $\bf S$ such that the $k$-dual of $\bf S$ is isomorphic with the right $A\ot A$-module $\JVHom_A({\bf \epsilon_A},{\bf \Delta_A})$, and such that $\bf S$ is free of rank one if it is considered as a $A$-$A^{op}$ bimodule.

\subsubsection{Topological Hopf algebras and locally compact quantum groups}

As generally known, the theory of Hopf algebras has known a great revival thanks to the discovery of quantum groups. Different from pure Hopf algebras, the theory of quantum groups is however not always purely algebraic but involves often topological and analytical features. This has lead to generalizations of the notion of a Hopf algebra.

First, one can consider a monoidal category of well-behaving topological vectorspaces with completed tensorproduct. A Hopf algebra in this category is called a topological Hopf algebra, see e.g. \cite{Larson} and \cite{TopHA}.

Motivated by the Pontryagin duality for locally compact topological commutative groups, Vaes and Kustermans \cite{KusVae} introduced locally compact quantum groups. To present day, a full Tannaka reconstruction theory for these locally compact quantum groups is not known to the author. 

\subsubsection{Combinations}

Of course, it is possible to make combinations of two or more of several of the structures that we have reviewed above. For example structures as weak quasi Hopf algebras, quasi-triangualar quasi Hopf algebras, weak group Hopf coalgebras, weak multiplier Hopf algebras have been investigated by several authors. 

\subsubsection{Further generalizations and prospects}
\selabel{prospects}

To combine all generalizations of Hopf algebras into a single and unifying framework, it became clear during the recent years that the proper setting will not be given by braided monoidal categories, not even by Hopf monads (as treated above). Rather one will have to move to a higher categorical setting, considering monoidal bicategories, 2-monoidal categories and possibly even more involved structures (see e.g. \cite{Chik}, \cite{BooStr}). These recent developments are indicating that the theory of Hopf algebras and the interrelation with (higher) category have an exciting future ahead.

\subsection*{Acknowledgement}
The author would like to thank Gabi B\"ohm for useful comments on an earlier version of the paper and Stef Caenepeel for discussions.

\end{document}